\documentclass{amsart}
\newtheorem{theorem}{Theorem}[section]

\usepackage{latexsym}
\usepackage{amssymb,amsmath,amsfonts}
\usepackage[colorlinks=true]{hyperref}
\usepackage{graphicx}
\usepackage[numbers,sort&compress]{natbib}
\usepackage{color}

\usepackage{float}
\usepackage{adjustbox}
\usepackage{longtable} 
\usepackage {setspace}
\usepackage{multicol, multirow}
\usepackage{bigstrut}
\usepackage{tabularx}
\usepackage{booktabs}

\newtheorem{lemma}[theorem]{Lemma}

\newtheorem{assup}[theorem]{Assumptions}
\newtheorem{algo}[theorem]{Algorithm}

\theoremstyle{definition}
\newtheorem{definition}[theorem]{Definition}

\newtheorem{example}[theorem]{Example}

\theoremstyle{remark}
\newtheorem{remark}[theorem]{Remark}

\numberwithin{equation}{section}
\global\parskip 6pt \oddsidemargin
0.1cm \evensidemargin 0.4cm \headheight 0.1cm \topmargin -2.0cm
\title[Proximal and Contraction method with Relaxed Inertial and Correction Terms...]
{Proximal and Contraction method with Relaxed Inertial and Correction Terms for Solving Mixed Variational Inequality Problems}
\author[\hspace{10.00cm}CE Nwakpa \textit{et al. ...}]{Chidi Elijah Nwakpa$^{1},$ Austine Efut Ofem$^{2}$, Kalu Okam Okorie$^{3},$ Chinedu Izuchukwu$^{4},$   Chibueze Christian Okeke$^{5}.$}
\address{$^{1,3,4,5}$School of Mathematics, University of the Witwatersrand, Private Bag 3, Johannesburg 2050, South Africa.}
\address{$^{2}$School of Mathematics, Statistics and Computer Sciences, University of KwaZulu-Natal, Durban, South Africa.}
\email{$^{1},$2713235@students.wits.ac.za, chidinecsc2n@gmail.com}
\email{$^{2},$ofemaustine@gmail.com}
\email{$^{3},$2710937@students.wits.ac.za}
\email{$^{4},$chinedu.izuchukwu@wits.ac.za}
\email{$^{5},$chibueze.okeke87@yahoo.com, chibueze.okeke@wits.ac.za}
\thanks{Corresponding author: Chibueze Christian Okeke.}

\begin{document}
		
	\begin{abstract}\noindent\\
	We propose in this paper a proximal and contraction method for solving a convex mixed variational inequality problem in a real Hilbert space. To accelerate the convergence of our proposed method, we incorporate an inertial extrapolation term, two correction terms, and a relaxation technique. We therefore obtain a weak convergence result under some mild assumptions. Finally, we present numerical examples to practically demonstrate the effectiveness of the relaxation technique, the inertial extrapolation term, and the correction terms in our proposed method.
	\end{abstract}
	
	\maketitle
	\section{Introduction}
	\noindent The theory of mixed variational inequalities in its various forms serves as a robust mathematical framework for modeling a wide range of phenomena in numerous research fields. The mixed variational inequality (briefly MVI) provides a unifying approach for modeling and finding approximate solutions to a wide range of linear and nonlinear operator problems encountered across numerous fields such as medical, engineering, and other science fields. In these fields, several mathematical problems formulated as Mixed variational inequality problems (MVIPs) have been applied to image restoration, electronics, fluid flow through media, financial analysis, decision making, ecology, and others \cite{Goeleven1,Goeleven2,Han,Konnov}.

	\noindent Given a nonempty convex set $\mathcal{C}$ in a Hilbert space $\mathcal{H},$ let $\langle\cdot , \cdot\rangle$ denote the inner product and  $\| \cdot\|$ its corresponding norm operator. Let $\mathcal{T}:\mathcal{H}\longrightarrow\mathcal{H}$ be a linear or nonlinear mapping and $g:\mathcal{C}\longrightarrow\overline{\mathbb{R}} = [-\infty, +\infty]$ a real-valued function. We consider the following problem:
	\begin{eqnarray}\label{MVI non}
		\begin{cases}
			\mbox{ Obtain} \ \bar{x}\in\mathcal{C} \ \mbox{such that} \\
			\langle \mathcal{T}\bar{x}, \underline{u}-\bar{x}\rangle+g(\underline{u})-g(\bar{x})\geq0 \ \ \forall \underline{u}\in\mathcal{C}.
		\end{cases}
	\end{eqnarray}
	This problem, described by \eqref{MVI non}, is called the mixed variational inequality problem (MVIP) or, alternatively, a variational inequality of the second kind.
	Let us denote $\Delta(\mathcal{T};g):=\{\bar{x}\in\mathcal{C} ~:~ \langle \mathcal{T}(\bar{x}), \underline{u}-\bar{x}\rangle +g(\underline{u})-g(\bar{x})\geq 0, \ \ \forall \underline{u}\in \mathcal{C}\}$ to be the set of solutions of the MVI\eqref{MVI non}. 
	  It is crucial to note  that the differentiability of the function $g$ in problem \eqref{MVI non} cannot be assumed. This is particularly evident when $g$ corresponds to the indicator function, $\iota_\mathcal{C},$ of a nonempty, closed set  $\mathcal{C}$ of $\mathcal{H}$, as such function often fails to satisfy the necessary regularity conditions for differentiability. In this particular case where 
	  \begin{eqnarray*}
	  	g(\underline{u})\equiv\iota_\mathcal{C}(\underline{u}) = \begin{cases}
	  		0, \ \ \mbox{if} \ \underline{u}\in\mathcal{C},\\
	  		\infty, \ \ \mbox{otherwise,} 
	  	\end{cases}
	  \end{eqnarray*}
	  the MVI\eqref{MVI non} becomes the classical variational inequality problem (VIP) of Stampacchia \cite{STAM}, which is:\\
	  obtain $\bar{x}\in\mathcal{C}$ such that 
	  \begin{eqnarray}\label{VIP1.2}
	  	\langle \mathcal{T}\bar{x}, \underline{u}-\bar{x}\rangle\geq 0 \ \ \mbox{for each} \ \underline{u}\in\mathcal{C}.
	  \end{eqnarray}
	  If $g$ is a proper, convex and  lower-semicontinuous function (its subdifferential $\partial g(\cdot)$ exists), then problem \eqref{MVI non} is equivalent to the problem of finding a zero of two monotone operators, described as follows:\\
	   Find $\bar{x}\in\mathcal{H}$ such that 
	  \begin{eqnarray}
	  	0\in(\mathcal{T}\bar{x}+\partial g(\bar{x})),
	  \end{eqnarray} 
	  where the subdifferential $\partial g$ of the proper, convex and  lower-semicontinuous function $g$ is known to be a maximal monotone operator.
	   Similarly, if $\mathcal{T}\equiv 0,$ MVI\eqref{MVI non} becomes the non-differentiable
	   convex optimization problem of minimizing $g$ over $\mathcal{C},$ defined as (see, \cite{Ricceri}):
	  \begin{eqnarray}
	  	\underset{\underline{u}\in\mathcal{C}}{\min}\{g(\underline{u})+\Phi(\underline{u})\}.
	  \end{eqnarray}
	  
	  \noindent Over the last few years, the MVIP has attracted significant research interest due to its wide range of applications. For example, Konnov and Volotskaya \cite[Section 2]{Konnov} reformulated the problem described in \eqref{MVI non} as an oligopolistic equilibrium model, which is a crucial concept in supply and demand market. Again, Goeleven \cite{Goeleven1,Goeleven2} demonstrated the applicability of the  MVI\eqref{MVI non} in the field of electrical circuits. As MVI generalizes numerous other models, a wide array of methods have been developed by various researchers to find an approximate solution to problem \eqref{MVI non}. Some of the most effective numerical techniques that have emerged over time are proximal methods and extragradient methods, among others (for example, see \cite{Garg,Han,JOLAOSO,Konnov,Nwakpa,Tang zhu}). However, one valuable tool stands out--the proximal operator, whose application has as well been found useful in addressing a wide range of other numerical problems. A key point to note is that the proximal operator is a strongly convex operator and hence, it guarantees the existence of a unique solution. Furthermore, problem \eqref{MVI non} is called a convex MVIP whenever the function $g$ is convex, and it is  best solved using the proximal type methods (see, for instance \cite{Chen,Wang}). 

	\noindent A significant challenge with MVI\eqref{MVI non} is developing efficient iterative algorithms and the analysis of their convergence. The presence of the nonlinear mapping $g$ makes this process considerably more complex for MVIP than for VIP\eqref{VIP1.2}. So, understanding that MVIP extends the concept of VIP, we can gain insight on how to solve MVI\eqref{MVI non} by reviewing solution techniques for VIP\eqref{VIP1.2}, particularly projection-type method.  The most fundamental version of the projection method is the Picard iteration given as:
	\begin{eqnarray*}
		x_{n+1} = \mathcal{P}_\mathcal{C}(x_n - \lambda\mathcal{T}x_n),
	\end{eqnarray*}
    where $\mathcal{P}_\mathcal{C}$ is a projection onto the feasible set $\mathcal{C},$ and $\lambda$ is a positive  real constant. However, this method requires the involved Lipschitz operator to be strongly monotone to guarantee convergence. This condition is quite restrictive. Thus, the extragradient method (EGM) introduced by Korpelevich \cite{Korpelevich} becomes crucial as it overcomes the strong monotonicity condition that the Lipschitz operator has to satisfy to guarantee convergence. This method given as 
    \begin{eqnarray*}
    	\begin{cases}
    		y_n = \mathcal{P}_\mathcal{C}(x_n-\lambda\mathcal{T}x_n),\\
    		x_{n+1} = \mathcal{P}_\mathcal{C}(x_n-\lambda\mathcal{T}y_n),
    	\end{cases}
    \end{eqnarray*}
    for each $n\in\mathbb{N},$ converges if the Lipschitz operator, $\mathcal{T},$ is monotone. However, the EGM came with another computational challenge of evaluating two projections onto the feasible set $\mathcal{C}$ in each iterative step. It is worth mentioning that these projection evaluations onto $\mathcal{C}$ could be computationally expensive especially when the convex set $\mathcal{C}$ is structurally complex.
	
	\noindent Furthermore, the computational drawback posed by the EGM can either be overcome by the subgradient extragradient method (SGEGM) or the projection (proximal in a case $g$ is not the indicator function) and contraction method (PCM). The SGEGM, proposed by Censor \textit{et al.} \cite{Censor Gibali 1,Censor Gibali 2}, replaces one of the projections onto $\mathcal{C}$ with a projection onto a carefully constructed half-space, thereby reducing the number of projection evaluations onto the feasible set $\mathcal{C}$ in every iterative step. On the other hand, the PCM, proposed by He \cite{He},  involves iteratively projecting a point onto a feasible set and then contracting the distance between the current point and its projection, often incorporating inertial techniques for improved convergence. This algorithm is particularly useful when dealing with monotone and Lipschitz continuous operators, although it has been extended to handle more complex scenarios like pseudomonotone operators.
	 The PCM defined as:
	\begin{eqnarray*}
		\begin{cases}
				y_n = \mathcal{P}_\mathcal{C}(x_n-\lambda_n\mathcal{T}x_n),\\
				x_{n+1} = x_n-\gamma\tau_n d(x_n, y_n),\\
				d(x_n, y_n) := x_n-y_n-\lambda_n(\mathcal{T}x_n-\mathcal{T}y_n),\\
				\tau_n = 
				\begin{cases}
					\dfrac{\langle x_n- y_n, d(x_n,y_n)\rangle} {\|d(x_n,y_n)\|^2}, \ \ \mbox{if} \ d(x_n,y_n) \neq 0,\\ \\
					0, \ \ \ \ \ \mbox{otherwise,}
				\end{cases}
				\end{cases}
			\end{eqnarray*}
	where $\gamma\in(0,2), \ \lambda_n\in\left(0, \dfrac{1}{\beta}\right)$ (or, alternatively,  $\lambda_n$ is updated using a self-adaptive rule), and $\beta>0$ is the Lipschitz constant of the monotone operator $\mathcal{T},$ requires only one projection evaluation in each iteration. Thus, owing to this computational advantage, the PCM has become of huge interest to many researchers (for example, see \cite{Alakoya,Cai,Xue Song,Zhang}) who have proposed different PCM for solving different optimization problems.

		\noindent  
	As new methods for solving MVI\eqref{MVI non} and other optimization problems  evolve, researchers are increasingly interested in exploring different techniques to achieve faster convergence of these methods. One approach is to incorporate inertial terms into these methods. The concept of inertial extrapolation terms in algorithms stems from Polyak's work \cite{Polyak} on the Heavy Ball with Friction (HBF), a second-order continuous system. The inertial extrapolation term arose naturally from an implicit time discretization of the HBF system.  A distinctive feature of this algorithm is that it utilizes the previous two iterates to determine the next iterate. Other inertial-type methods can be found in \cite{Dong Cho,Tan Li,Tan Qin} and references quoted in these works.

	\noindent 
    Additionally, another technique for accelerating algorithms for solving optimization problems involves incorporating correction terms. On this note, Kim \cite{Kim} in his work proposed an accelerated proximal point method that combines the proximal point method with inertial and correction terms. Although no weak convergence result was obtained for the generated sequences, Kim \cite[Theorem 4.1]{Kim} leveraged the performance estimation problem (PEP) approach of Drori and Teboulle \cite{Drori} to obtain the worst-case convergence rate $\Big(\|y_n-w_{n-1}\|=\mathcal{O}(n^{-1})\Big)$ for the following algorithm: 
	\begin{eqnarray}
		\begin{cases}
			w_n=y_n+\dfrac{n-1}{n+1}(y_n-y_{n-1})+\dfrac{n-1}{n+1}(w_{n-2}-y_{n-1}),\\
			\\
			y_{n+1}=J^A_\lambda(w_n).
		\end{cases}
	\end{eqnarray}
	In another related work, Maingé \cite{Maingé} further investigated the proximal point method by combining inertial, relaxation, and correction terms, with a focus on solving monotone inclusion problems. He proposed the following method
	\begin{eqnarray}
		\begin{cases}
			w_n=y_n+\alpha_n(y_n-y_{n-1})+\delta_n(w_{n-1}-y_n),\\
			\\
			y_{n+1}=\dfrac{1}{1+\alpha_n}w_n+\dfrac{\alpha_n}{1+\alpha_n}J^A_{\lambda(1+\alpha_n)}(w_n),
		\end{cases}
	\end{eqnarray}
	and established weak convergence results and the fast rate, $\|y_{n+1}-y_n\|=o(n^{-1}),$ for the generated sequence. Motivated by the works of Kim \cite{Kim} and Maingé \cite{Maingé}, Izuchukwu \textit{et al.} \cite{Izuchukwu} proposed another method for solving proximal point problems. Their method involves incorporating two correction terms, yielding an approach that can be described as:
	\begin{eqnarray}
		\begin{cases}
			w_n=y_n+\alpha(y_n-y_{n-1})+\delta(1+\alpha)(w_{n-1}-y_n)-\alpha\delta(w_{n-2}-y_{n-1}),\\
			\\
			y_{n+1} = J^A_\lambda(w_n).
		\end{cases}
	\end{eqnarray}
	In \cite[Theorem 3.5]{Izuchukwu}, they obtained a weak convergence result for their proposed method, while in \cite[Theorem 3.7]{Izuchukwu}, they established a linear convergence rate, both under certain assumptions. Through numerical examples, they further revealed that the incorporation of multiple correction terms significantly accelerates their method, outperforming those in \cite{Kim,Maingé} that rely on a single correction term.
	
	\noindent Building on the works of Dong \textit{et al.} \cite{Dong Cho}, Izuchukwu \textit{et al.} \cite{Izuchukwu}, Kim \cite{Kim}, and Maingé \cite{Maingé}, this work proposes a proximal and contraction method with relaxed inertial and correction terms for solving mixed variational inequality problems. The key contributions of this approach include:
	\begin{itemize}
		\item Combining one inertial term, two correction terms with a relaxation technique to propose a proximal and contraction method for solving MVIPs in a real Hilbert space;
		\item  Incorporating a self-adaptive stepsize, which is distinct from the methods proposed in \cite{Dong Cho,Tang zhu} that use constant stepsizes, and \cite{Alakoya,Cholamjiak Thong} that use the on-line rule and/or the line search rule;
		\item Obtaining a weak convergence result and providing numerical examples to justify the effectiveness of our proposed method.
	\end{itemize}

	 \noindent The rest of the content of this work is arranged as follows:  In Section \textbf{\ref{pre}}, we present some important definitions, lemmas and preliminary
	 results that are subsequently needed in this work. In Section \textbf{\ref{MAIN RESULT}}, we present and discuss our proposed method for solving a convex MVIP and obtain a weak convergence result. Section \textbf{\ref{numerics}} presents some numerical results which serve as practical illustrations of the effectiveness of involving relaxation technique, inertial extrapolation term and two correction terms in our proposed method. 
	 
	\section{Preliminaries}\label{pre}
	\noindent In this section, we give some basic definitions and lemmas that are necessary for the convergence analysis of our proposed system. We shall denote $\overline{\mathbb{R}}:=\mathbb{R}\cup\{-\infty, +\infty\}$ to be the extended real number.\\
	
\begin{definition}
	Let $g:\mathcal{H}\longrightarrow\overline{\mathbb{R}}$ be a real-valued function. Then
	\begin{enumerate}
		\item [(i)] the effective domain of $g$ is defined by $dom~g:=\{\underline{u}\in\mathcal{H} \ \vert \ g(\underline{u})<+\infty\};$
		\item [(ii)] $g$ is said to be a proper function if
		its effective domain is non-empty, i.e., there exists at least one $\underline{u}\in\mathcal{H}$ such that $g(\underline{u})<+\infty$ and
		 if $g(\underline{u})>-\infty$ for any $\underline{u}\in \mathcal{H}.$  Thus, one can also verify that  $g(\underline{u})=+\infty$ for any $\underline{u}\notin dom~g;$
		\item [(iii)] $g$ is known as a convex function if its domain is convex and for any $\underline{u},\underline{v}\in dom~g,$ it follows that
		\begin{eqnarray*}
			g(t\underline{u}+(1-t)\underline{v})\leq t g(\underline{u}) + (1-t)g(\underline{v}), \ \ \ \forall t\in[0,1].
		\end{eqnarray*}
	\end{enumerate}
\end{definition}

	\begin{definition}
		Let $g:\mathcal{C}\subset\mathcal{H}\longrightarrow\overline{\mathbb{R}}$ be a proper lower semi-continuous real-valued function, where $\mathcal{C}$ is a nonempty closed convex set. For every $\underline{u}, \underline{v} \in\mathcal{C},$ the mapping $\mathcal{T} : \mathcal{H}\longrightarrow\mathcal{H}$ is said to be 
		\begin{enumerate}
			\item [(i)] Lipschitz continuous if there is a constant $\beta>0,$ called the modulus (or Lipschitz constant), such that
			\begin{eqnarray*}
				\|\mathcal{T}\underline{u}-\mathcal{T}\underline{v}\|\leq\beta\|\underline{u}-\underline{v}\|;
			\end{eqnarray*} 
			\item [(ii)] monotone if 
			\begin{eqnarray*}
			\langle \mathcal{T}\underline{u}-\mathcal{T}\underline{v}, \underline{u}-\underline{v}\rangle\geq0;	
			\end{eqnarray*}
			\item [(iii)] strongly monotone if there is a constant $\kappa>0$ such that
			\begin{eqnarray*}
				\langle \mathcal{T}\underline{u}-\mathcal{T}\underline{v}, \underline{u}-\underline{v}\rangle\geq\kappa\|\underline{u}-\underline{v}\|^2;
			\end{eqnarray*}
			\item [(iv)] pseudomonotone if  
			\begin{eqnarray*}
				\langle\mathcal{T}\underline{u}, \underline{v}-\underline{u}\rangle\geq0 \implies \langle\mathcal{T}\underline{v}, \underline{v}-\underline{u}\rangle\geq0;
			\end{eqnarray*}
			\item [(v)] $g-$pseudomonotone on $\mathcal{C},$ if 
			\begin{eqnarray*}
				\langle\mathcal{T}\underline{u}, \underline{v}-\underline{u}\rangle +g(\underline{v}) - g(\underline{u})\geq0 \implies \langle\mathcal{T}\underline{v}, \underline{v}-\underline{u}\rangle +g(\underline{v}) - g(\underline{u})\geq0 \ \  \mbox{(See \cite{JOLAOSO})};
			\end{eqnarray*}
		\end{enumerate}
	\end{definition}
	\begin{remark}\noindent
		\begin{enumerate}
			\item One can easily verify that $\kappa\leq\beta$ if $\mathcal{T}$ is $\kappa-$strongly monotone and $\beta-$Lipschitz continuous on $\mathcal{C}.$
			\item Note that the $g-$pseudomonotonicity of $\mathcal{T}$ on the $dom~g$ does not necessary mean that $\mathcal{T}$ is pseudomonotone. This can be easily checked using Example \ref{ex1} below. However, if $g=\iota_\mathcal{C},$ then the $0-$pseudomonotonicity of $\mathcal{T}$ corresponds to the pseudomonotonicity condition of $\mathcal{T}.$
		\end{enumerate}
	\end{remark}
	\begin{example}\label{ex1}
		Let $\mathcal{C}=[3,5].$ Define $g:\mathcal{C}\longrightarrow\mathbb{R}$ and $\mathcal{T}:\mathbb{R}\longrightarrow\mathbb{R}$ by\\
		\begin{eqnarray*}
			g(\underline{u})=
			\begin{cases}
				\underline{u}^2, \ \ \mbox{if} \ \underline{u}\in\mathcal{C},\\
				+\infty, \ \ \mbox{otherwise}
			\end{cases}
		\end{eqnarray*}
		and 
		$$\mathcal{T}(\underline{u})=4-\underline{u} \  \ \ \ \forall \underline{u}\in\mathbb{R}.$$
		Let the inner product function $\langle  \cdot, \cdot\rangle$ on $\mathbb{R}$ be defined by $\langle \underline{u},\underline{v}\rangle =\underline{u} \cdot\underline{v}$ for each $\underline{u},\underline{v}\in\mathbb{R}.$ Consider the MVI of finding $\bar{x}\in dom \ g$ such that
		$$\langle \mathcal{T}\underline{u}, \underline{v}-\underline{u}\rangle+g(\underline{v})-g(\underline{u})\geq0, \ \ \forall \underline{v}\in dom~g.$$ \\
		In what that follows, we show that B$_1-$B$_3$ of Assumption \ref{ASSUMproxcon} below are satisfied, thus:
		\begin{enumerate}
			\item It can easily be verified that the solution set $\Delta(\mathcal{T};g)$ is nonempty. In fact, $\bar{x}=3$ is the unique solution of the MVI.
			\item It is clear that $\mathcal{T}$ is $1-$Lipschitz continuous. It can also be readily observed that the mapping $\mathcal{T}$ is sequentially weakly continuous.
			\item Clearly, the function $g$ is proper and convex. In addition, since $|g(\underline{u})-g(\underline{v})| = |\underline{u}^2-\underline{v}^2|\leq|\underline{u}+\underline{v}||\underline{u}-\underline{v}|\leq10|\underline{u}-\underline{v}|,$ thus, $g$ is Lipschitz continuous with constant $10.$ Now, we show that the $g-$pseudomonotonicity of $\mathcal{T}$ does not necessarily guarantee the pseudomonotonicity of $\mathcal{T}.$ \\
			Suppose 
			$$\langle \mathcal{T}\underline{u}, \underline{v}-\underline{u}\rangle+g(\underline{v})-g(\underline{u})\geq0, \ \ \forall \underline{v}\in [3,5].$$
			That is,
			$$(\underline{v}-\underline{u})(4+\underline{v})\geq0.$$
			But $4+\underline{v}>0,$ since $\underline{v}\in[3,5].$ This further implies that $(\underline{v}-\underline{u})\geq0.$ Hence, we have that 
			\begin{eqnarray*}
				\langle \mathcal{T}\underline{v}, \underline{v}-\underline{u}\rangle+g(\underline{v})-g(\underline{u})
				&=& (\underline{v}-\underline{u})(4+\underline{u})\\
				&\geq&0,
			\end{eqnarray*} 
			which shows that $\mathcal{T}$ is $g-$pseudomonotone on $dom~g.$ However, if we take $\underline{u} = 3$ and $\underline{v} = 5$ then, it follows that $\langle \mathcal{T}\underline{u}, \underline{v}-\underline{u}\rangle = 2>0,$ and $\langle \mathcal{T}\underline{v}, \underline{v}-\underline{u}\rangle = -2<0,$ showing that $\mathcal{T}$ is not pseudomonotone.
		\end{enumerate}
	\end{example}

\begin{lemma}\label{LM con1}
	The following hold for every $\underline{u},\underline{v}\in \mathcal{H}$ and $\sigma \in \mathbb{R}:$
	\begin{enumerate}
		\item [(a)]
		 $\|(1+\sigma) \underline{u} -\sigma\underline{v}\|^2 = (1+\sigma)\|\underline{u}\|^2 - \sigma\|\underline{v}\|^2 + \sigma(1+\sigma)\|\underline{u}-\underline{v}\|^2;$
		 \item [(b)] $\langle \underline{u}- \underline{w}, \underline{v}-\underline{u}\rangle=\dfrac{1}{2}\|\underline{w}-\underline{v}\|^2-\dfrac{1}{2}\|\underline{u}-\underline{w}\|^2-\dfrac{1}{2}\|\underline{u}-\underline{v}\|^2.$ 
		 	\end{enumerate}
\end{lemma}
\begin{lemma}\cite{BAUS}\label{prox lemma}
	For every $\underline{u}\in\mathcal{H},$ $\underline{v}\in\mathcal{C},$ and $\lambda>0$ a scalar, the following inequality holds\\
	$$\lambda(g(\underline{v})-g(prox_{\lambda g}(\underline{u})))\geq \langle \underline{u}- prox_{\lambda g}(\underline{u}), \underline{v}-prox_{\lambda g}(\underline{u})\rangle,$$ \\
	where $prox_{\lambda g}(\underline{u}) := \underset{\underline{v}\in\mathcal{C}}{argmin}\left\{\lambda g(\underline{v})+\frac{\|\underline{u}-\underline{v}\|^2}{2}\right\}.$
\end{lemma}
\noindent The result stated in the following lemma can be readily derived from Lemma \ref{prox lemma}.
\begin{lemma}\label{prox lemma result}
Let $g$ be a proper, lower semi-continuous, and convex function, and $\lambda$ a positive constant. Then an element $r^*\in\mathcal{H}$ solves MVI\eqref{MVI non} if and only if 
$$r^* = prox_{\lambda g}(r^*-\lambda\mathcal{T}r^*).$$	
\end{lemma}

\begin{lemma}\cite{Opial} \label{Opial Lemma}
	Let $\mathcal{C}\subset \mathcal{H}$ be nonempty, and $\{x_n\}$ a sequence in $\mathcal{H}$ such that:
	\begin{itemize}
		\item[(i)]  $\underset{n\to \infty}\lim\|x_n - \bar{r}\|$ exists for each $\bar{r}\in \mathcal{C}$;
		\item[(ii)] every sequentially weak cluster point of $\{x_n\}$ is found in $\mathcal{C}.$\\ \\ Then the sequence $\{x_n\}$ weakly converges to a point in $\mathcal{C}.$
	\end{itemize} 
\end{lemma}

\hfill

\hfill

	\section{Main Results}\label{MAIN RESULT}
\noindent Before we present our method, let us first consider some assumptions which the convergence of the generated sequences is based on.

\begin{assup}\label{ASSUMproxcon}
	\
	\noindent\\ Suppose the following hold:
	\begin{enumerate}
		\item [(B$_1$)] The solutions set, $\Delta(\mathcal{T};g)\neq\emptyset;$
		\item [(B$_2$)] The Lipschitz continuous operator $\mathcal{T}$ with Lipschitz constant $\beta>0$ is sequentially weakly continuous;
		\item [(B$_3$)] $g$ is a convex lower semicontinuous real-valued function on $\mathcal{C};$
		\item [(B$_4$)] The operator $\mathcal{T}$ is monotone. In particular, $\mathcal{T}$ and $g$ satisfy the following generalized monotonicity condition on $\mathcal{C}$
		\begin{eqnarray*}
			\langle \mathcal{T}(\underline{v}), \underline{v}-\underline{u}\rangle +g(\underline{v})-g(\underline{u})\geq0, \ \ \forall \underline{v}\in\mathcal{C}, \ \ \forall \underline{u}\in \Delta(\mathcal{T};g).
		\end{eqnarray*}
	\end{enumerate}
\end{assup}

\begin{assup}\label{Assumconditions}
	\
	\noindent\\ Let $\theta\in(0,1), \ \gamma\in(0,2), \ \sigma>0, \xi>0,$ where $\xi :=\dfrac{1}{\theta}\left(\dfrac{2-\gamma}{\gamma}+1-\theta\right) $. Let $\alpha\in[0,1)$ and $\delta\in(0,1)$ satisfy the following conditions:
	\begin{enumerate}
		\item [(a)] $$0\leq\alpha<\dfrac{\sigma}{1+\sigma};$$
		\item [(b)] 	$$\max\left\{\dfrac{\alpha(1+\sigma)}{1+\alpha\sigma}, \ \ \dfrac{\alpha(1+\alpha)+2\alpha\xi+1-\sqrt{\alpha^4+2\alpha^3+3\alpha^2+4\alpha\xi+2\alpha+1}}{2\alpha\xi}\right\}<\delta.$$
 	\end{enumerate} 
\end{assup}
Now, we present our method as follows:

\hrule
\begin{algo}\label{ALG proxcon}
Proximal and Contraction Method with Relaxed Inertial and Correction Terms
\hrule
\begin{itemize}
	\item [\textbf{STEP 1:}] Select $\alpha$ and $\delta$ such that Assumption \ref{Assumconditions} is satisfied. Pick $w_{-1} = w_{-2}, \ x_0, \ x_{-1}\in\mathcal{H}$ arbitrarily. Let $\lambda_0>0, \ \mu\in (0,1), \ \theta\in(0,1)$ and $\gamma\in(0,2),$ and set $n=0.$
	\item [\textbf{STEP 2:}] Given the current iterates $w_{n-1},$ $w_{n-2},$ $x_{n-1}$ and $x_n,$ evaluate
\begin{eqnarray*}
	\begin{cases}
	w_n = x_n+\alpha(x_n-x_{n-1}) +\delta(1+\alpha)(w_{n-1}-x_n)-\alpha\delta(w_{n-2}-x_{n-1}),\\ \\
	y_n=\mbox{prox}_{\lambda_n g}(w_n-\lambda_n\mathcal{T}w_n),\\ \\ 
	d(w_n,y_n) = (w_n-y_n)-\lambda_n(\mathcal{T}w_n-\mathcal{T}y_n),\\ \\
	z_n = w_n - \gamma\tau_n d(w_n, y_n),
	\end{cases}
\end{eqnarray*}
where 
\begin{eqnarray*}
	\tau_n = 
	\begin{cases}
		\dfrac{\Pi(w_n, y_n)}{\|d(w_n,y_n)\|^2}, \ \ \mbox{if} \ d(w_n,y_n) \neq 0,\\ \\
		0, \ \ \ \ \ \mbox{otherwise,}
	\end{cases}
\end{eqnarray*}
where $\ \ \ \Pi(w_n, y_n) = \langle w_n-y_n, d(w_n,y_n)\rangle.$
\item [\textbf{STEP 3:}] Update the stepsize as:
\begin{eqnarray*}
	\lambda_{n+1} = 
	\begin{cases}
		\min\left\{\dfrac{\mu\|w_n-y_n\|}{\|\mathcal{T}w_n-\mathcal{T}y_n\|}, \lambda_n\right\}, \ \ \mbox{if} \ \mathcal{T}w_n\neq\mathcal{T}y_n, \\ \\
		\lambda_n, \ \ \ \ \ \ \ \ \ \ \ \mbox{otherwise.}
	\end{cases}
\end{eqnarray*} 
\item [\textbf{STEP 4:}] Evaluate
\begin{eqnarray*}
	x_{n+1}= (1-\theta)w_n+\theta z_n.
\end{eqnarray*}

\noindent Set $n:=n+1,$ and go back to \textbf{ STEP 2}.
\end{itemize}
\end{algo}

\begin{remark}\noindent
		\begin{itemize}
		\item 	Notice that, for each $n\geq0,$ the stepsize $\lambda_n$ defined in \textbf{step 3} is a non-increasing monotone sequence that is bounded below. Thus, $\underset{n\to\infty}{\lim}\lambda_n = \lambda^*>0$ exists.
		\item A notable advantage of our proposed Algorithm \ref{ALG proxcon} is that our method incorporates inertial, a two-term correction and relaxation techniques. In fact, methods like the ones studied in \cite{Dong Cho,Wang Chen} can be recovered from our method if $\delta$
        and $\theta$ are set to be zero.
	\end{itemize}

\end{remark}
\begin{lemma}
	Given that $y_n, w_n$ and $d(w_n, y_n)$ are as defined in Algorithm \ref{ALG proxcon}. If $y_n = w_n$ or $d(w_n, y_n) = 0 $ for each $n\geq1,$ then $x_{n+1}\in\Delta(\mathcal{T}; g).$
\end{lemma}
\begin{proof}\noindent\\
	\noindent Using the fact that $\mathcal{T}$ is $\beta-$Lipschitz we have, $\forall n\in\mathbb{N},$ that 
	\begin{eqnarray}\label{s1}
		\|d(w_n, y_n)\| &=& \|w_n-y_n-\lambda_n(\mathcal{T}w_n-\mathcal{T}y_n)\|\nonumber\\
		&\geq&\|w_n-y_n\|-\lambda_n\|\mathcal{T}w_n-\mathcal{T}y_n\|\nonumber\\
		&\geq& \|w_n-y_n\|-\lambda_n\beta\|w_n-y_n\|\nonumber\\
		&=& (1-\lambda_n\beta)\|w_n-y_n\|.
	\end{eqnarray}
	Again, one can quickly see that 
	\begin{eqnarray}\label{s2}
		\|d(w_n,y_n)\|\leq(1+\lambda_n\beta)\|w_n-y_n\|, \ \ \forall n\in\mathbb{N}.
	\end{eqnarray}
	Combining \eqref{s1} and \eqref{s2}, we get
	\begin{eqnarray*}
		(1-\lambda_n\beta)\|w_n-y_n\|\leq 	\|d(w_n,y_n)\|\leq(1+\lambda_n\beta)\|w_n-y_n\|, \ \ \forall n\geq1,
	\end{eqnarray*}
	so that $w_n =y_n$ if and only if $d(w_n, y_n)=0.$ Thus, 
	$$y_n = prox_{\lambda_ng}(y_n-\lambda_n\mathcal{T}y_n+ d(w_n,y_n)) \ \ \forall n\geq1.$$
	When $d(w_n,y_n)=0,$ using the definition of $z_n,$ it follows that $z_n = w_n.$ Also, using the definition of $x_{n+1}$, this further implies that $x_{n+1} = w_n=y_n$ and
	$$y_n = prox_{\lambda_ng}(y_n-\lambda_n\mathcal{T}y_n) \ \ \forall n\geq1.$$
	Therefore, by Lemma \ref{prox lemma result} $x_{n+1}\in\Delta(\mathcal{T};g).$
\end{proof}
\begin{lemma}
	Let $\{x_n\}$ be the sequence generated by Algorithm \ref{ALG proxcon}. If $\bar{r}\in
	\Delta(\mathcal{T};g),$ then, under B$_1$ and B$_4$ of Assumption \ref{ASSUMproxcon}, we have
	\begin{eqnarray}\label{LMPA}
		&\|x_{n+1}-\bar{r}\|^2\leq\|w_n-\bar{r}\|^2-\xi\|x_{n+1}-w_n\|^2,&\\
		&\mbox{where} \ \ \xi:=\dfrac{1}{\theta}\left(\dfrac{2-\gamma}{\gamma}+1-\theta\right)&\nonumber
	\end{eqnarray} 
	and
	
	\begin{eqnarray}\label{LMPB}
		\|w_n-y_n\|\leq\dfrac{1}{\theta\gamma}\left(\dfrac{1+\mu\dfrac{\lambda_n}{\lambda_{n+1}}}{1-\mu\dfrac{\lambda_n}{\lambda_{n+1}}}\right)\|x_{n+1}-w_n\|.
	\end{eqnarray}
\end{lemma}
\begin{proof}\noindent\\
	\noindent By the definition of $x_{n+1},$ we obtain
	\begin{eqnarray}\label{1a}
		\|x_{n+1}-\bar{r}\|^2 &=& \|(1-\theta)w_n+\theta z_n-\bar{r}\|^2\nonumber\\
		&=&\|(1-\theta)(w_n-\bar{r})+\theta(z_n-\bar{r})\|^2\nonumber\\
		&=& (1-\theta)\|w_n-\bar{r}\|^2+\theta\|z_n-\bar{r}\|^2-\theta(1-\theta)\|z_n-w_n\|^2.
	\end{eqnarray}
	From the definition of $z_n,$ it follows that
	\begin{eqnarray}\label{5}
		\|z_n-\bar{r}\|^2 &=& \|(w_n-\bar{r})-\gamma\tau_n d(w_n, y_n)\|^2\nonumber\\
		&=& \|w_n-\bar{r}\|^2-2\gamma\tau_n\langle w_n-\bar{r}, d(w_n,y_n)\rangle + \gamma^2\tau_n^2\|d(w_n,y_n)\|^2.
	\end{eqnarray}
	By combining the definition $y_n$ and Lemma \ref{prox lemma}, then
	\begin{eqnarray}\label{5prox}
		\lambda_n(g(y_n)-g(z))\leq\langle y_n-z, w_n-y_n-\lambda_n\mathcal{T}w_n\rangle \ \ \ \forall z\in\mathcal{C}.
	\end{eqnarray}
	Putting $z=\bar{r}\in\Delta(\mathcal{T}; g)$ into \eqref{5prox}, we have
	\begin{eqnarray}\label{6}
		\lambda_n(g(y_n)-g(\bar{r}))\leq\langle y_n-\bar{r}, w_n-y_n-\lambda_n\mathcal{T}w_n\rangle.
	\end{eqnarray}
	Since B$_4$ of Assumption \ref{ASSUMproxcon} is fulfilled, we have
	\begin{eqnarray*}
		g(\bar{r})-g(y_n)\leq\langle y_n-\bar{r}, \mathcal{T}y_n\rangle.
	\end{eqnarray*}
	That is,
	\begin{eqnarray}\label{7}
		\lambda_n(g(\bar{r})-g(y_n))\leq\langle y_n-\bar{r}, \lambda_n\mathcal{T}y_n\rangle.
	\end{eqnarray}
	Adding \eqref{6} and \eqref{7} yields
	\begin{eqnarray*}
		\langle y_n-\bar{r}, d(w_n,y_n)\rangle = \langle y_n-\bar{r}, (w_n-y_n)-\lambda_n(\mathcal{T}w_n-\mathcal{T}y_n)\rangle\geq0.
	\end{eqnarray*}
	On the other hand,
	\begin{eqnarray}\label{8}
		\langle w_n-\bar{r}, d(w_n,y_n)\rangle &=& \langle w_n-y_n, d(w_n,y_n)\rangle + \langle y_n-\bar{r}, d(w_n,y_n)\rangle\nonumber\\
		&\geq& \langle w_n-y_n, d(w_n,y_n)\rangle\nonumber\\
		&=& \Pi(w_n,y_n).
	\end{eqnarray}
	Using \eqref{8} in \eqref{5} and noting that $\tau_n = \dfrac{\Pi(w_n,y_n)}{\|d(w_n,y_n)\|^2}$ then, we get
	\begin{eqnarray}\label{9}
		\|z_n-\bar{r}\|^2&\leq&\|w_n-\bar{r}\|^2-2\gamma\tau_n\Pi(w_n,y_n)+\gamma^2\tau_n^2\|d(w_n,y_n)\|^2\nonumber\\
		&=&\|w_n-\bar{r}\|^2-\gamma(2-\gamma)\tau_n\Pi(w_n, y_n).
	\end{eqnarray}
	Again, from the definition of $z_n,$ it follows that
	\begin{eqnarray}\label{10}
		\tau_n\Pi(w_n,y_n) = \|\tau_n d(w_n,y_n)\|^2=\dfrac{1}{\gamma^2}\|z_n-w_n\|^2.
	\end{eqnarray}
	Using \eqref{10} in \eqref{9}, we have
	\begin{eqnarray}\label{11}
		\|z_n-\bar{r}\|^2\leq\|w_n-\bar{r}\|^2-\dfrac{2-\gamma}{\gamma}\|z_n-w_n\|^2.
	\end{eqnarray}
	Applying \eqref{11} in \eqref{1a}, we obtain
	\begin{eqnarray*}
		\|x_{n+1}-\bar{r}\|^2&\leq&(1-\theta)\|w_n-\bar{r}\|^2+\theta\left[\|w_n-\bar{r}\|^2-\dfrac{2-\gamma}{\gamma}\|z_n-w_n\|^2\right]-\theta(1-\theta)\|z_n-w_n\|^2\\
		&=& \|w_n-\bar{r}\|^2-\theta\left(\dfrac{2-\gamma}{\gamma}+1-\theta\right)\|z_n-w_n\|^2.
	\end{eqnarray*}
	But $x_{n+1} = (1-\theta)w_n+\theta z_n$ implies that $z_n-w_n = \dfrac{1}{\theta}(x_{n+1}-w_n).$\\
	
	 \noindent Therefore
	 \begin{eqnarray*}
	 	\|x_{n+1}-\bar{r}\|^2&\leq&\|w_n-\bar{r}\|^2-\dfrac{1}{\theta}\left(\dfrac{2-\gamma}{\gamma}+1-\theta\right)\|x_{n+1}-w_n\|^2\\
	 	&=& \|w_n-\bar{r}\|^2-\xi\|x_{n+1}-w_n\|^2,
	 \end{eqnarray*}
	where $\xi = \dfrac{1}{\theta}\left(\dfrac{2-\gamma}{\gamma}+1-\theta\right).$ Hence, establishing \eqref{LMPA}.
	
	\noindent Furthermore,
	\begin{eqnarray}\label{13}
		\Pi(w_n,y_n) &=& \langle w_n-y_n, d(w_n,y_n)\rangle\nonumber\\
		&=& \langle w_n-y_n, (w_n-y_n)-\lambda_n(\mathcal{T}w_n-\mathcal{T}y_n)\rangle\nonumber\\
		&=&\|w_n-y_n\|^2-\lambda_n\langle w_n-y_n, \mathcal{T}w_n-\mathcal{T}y_n\rangle\nonumber\\
		&\geq& \|w_n-y_n\|^2-\lambda_n\|w_n-y_n\|\|\mathcal{T}w_n-\mathcal{T}y_n\|\nonumber\\
		&\geq& \|w_n-y_n\|^2-\mu\dfrac{\lambda_n}{\lambda_{n+1}}\|w_n-y_n\|^2\nonumber\\
		&=&\left(1-\mu\dfrac{\lambda_n}{\lambda_{n+1}}\right)\|w_n-y_n\|^2.
	\end{eqnarray}
	It follows from \eqref{13} and the definitions of $z_n$ and $x_{n+1}$ that
	\begin{eqnarray*}
		\|w_n-y_n\|^2&\leq&\left(\dfrac{1}{1-\mu\dfrac{\lambda_n}{\lambda_{n+1}}}\right)\Big\langle w_n-y_n, d(w_n, y_n)\Big\rangle\\
		&=&\left(\dfrac{1}{1-\mu\dfrac{\lambda_n}{\lambda_{n+1}}}\right)\Pi(w_n,y_n)\\
		&=& \left(\dfrac{1}{1-\mu\dfrac{\lambda_n}{\lambda_{n+1}}}\right)\tau_n\|d(w_n,y_n)\|^2\\
		&\leq& \left(\dfrac{1}{1-\mu\dfrac{\lambda_n}{\lambda_{n+1}}}\right)\tau_n\|d(w_n,y_n)\|\Big(\|w_n-y_n\|+\lambda_n\|\mathcal{T}w_n-\mathcal{T}y_n\|\Big)\\
		&\leq& \left(\dfrac{1}{1-\mu\dfrac{\lambda_n}{\lambda_{n+1}}}\right)\tau_n\|d(w_n,y_n)\|\Big(\|w_n-y_n\|+\mu\dfrac{\lambda_n}{\lambda_{n+1}}\|w_n-y_n\|\Big)\\
		&=&\left(\dfrac{1+\mu\dfrac{\lambda_n}{\lambda_{n+1}}}{1-\mu\dfrac{\lambda_n}{\lambda_{n+1}}}\right)\tau_n\|d(w_n,y_n)\|\|w_n-y_n\|\\
		&\leq&\dfrac{1}{\gamma}\left(\dfrac{1+\mu\dfrac{\lambda_n}{\lambda_{n+1}}}{1-\mu\dfrac{\lambda_n}{\lambda_{n+1}}}\right)\|w_n-z_n\|\|w_n-y_n\|\\
		&=&\dfrac{1}{\theta\gamma}\left(\dfrac{1+\mu\dfrac{\lambda_n}{\lambda_{n+1}}}{1-\mu\dfrac{\lambda_n}{\lambda_{n+1}}}\right)\|x_{n+1}-w_n\|\|w_n-y_n\|,
	\end{eqnarray*}
	and \eqref{LMPB} follows immediately.
\end{proof}
\begin{lemma}
	Suppose Assumptions \ref{ASSUMproxcon} and \ref{Assumconditions} are satisfied. Then the sequences generated by Algorithm \ref{ALG proxcon} are bounded.
\end{lemma}
\begin{proof}\noindent\\
	\noindent If we let $s_n=x_n+\delta(w_{n-1}-x_n)$ then, it ensues from the definition of $w_n$ in Algorithm \ref{ALG proxcon} that
	$$w_n=s_n+\alpha(s_n-s_{n-1}).$$
	Furthermore, we obtain from the definition of $s_n$ that
	$$x_n = \dfrac{1}{1-\delta}s_n-\dfrac{\delta}{1-\delta}w_{n-1},$$
	and
	\begin{eqnarray}\label{14}
		x_{n+1} = \dfrac{1}{1-\delta}s_{n+1}-\dfrac{\delta}{1-\delta}w_n.
	\end{eqnarray}
	By Lemma \ref{LM con1} (a) and \eqref{14}, we obtain
	\begin{eqnarray}\label{15}
		\|x_{n+1}-\bar{r}\|^2 &=& \|\dfrac{1}{1-\delta}(s_{n+1}-\bar{r})-\dfrac{\delta}{1-\delta}(w_n-\bar{r})\|^2\nonumber\\
		&=&\dfrac{1}{1-\delta}\|s_{n+1}-\bar{r}\|^2-\dfrac{\delta}{1-\delta}\|w_n-\bar{r}\|^2+\dfrac{\delta}{(1-\delta)^2}\|s_{n+1}-w_n\|^2.	\end{eqnarray}
		Substituting \eqref{15} in \eqref{LMPA}, we have
	\begin{eqnarray*}
		\dfrac{1}{1-\delta}\|s_{n+1}-\bar{r}\|^2-\dfrac{\delta}{1-\delta}\|w_n-\bar{r}\|^2+\dfrac{\delta}{(1-\delta)^2}\|s_{n+1}-w_n\|^2\leq\|w_n-\bar{r}\|^2-\xi\|x_{n+1}-w_n\|^2.
	\end{eqnarray*}
	Or equivalently,
	\begin{eqnarray}\label{16}
		\dfrac{1}{1-\delta}\|s_{n+1}-\bar{r}\|^2\leq	\dfrac{1}{1-\delta}\|w_n-\bar{r}\|^2 -\xi\|x_{n+1}-w_n\|^2 - \dfrac{\delta}{(1-\delta)^2}\|s_{n+1}-w_n\|^2.
	\end{eqnarray}
	On the other hand,
	\begin{eqnarray}\label{17}
		\|x_{n+1}-w_n\|^2 &=& \|x_{n+1}-s_n-\alpha(s_n-s_{n-1})\|^2\nonumber\\
		&=&\|x_{n+1}-s_n\|^2-2\alpha\langle x_{n+1}-s_n, s_n-s_{n-1}\rangle+\alpha^2\|s_n-s_{n-1}\|^2\nonumber\\
		&\geq&\|x_{n+1}-s_n\|^2-2\alpha\|x_{n+1}-s_n\|\|s_n-s_{n-1}\|+\alpha^2\|s_n-s_{n-1}\|^2\nonumber\\
		&\geq& \|x_{n+1}-s_n\|^2-\alpha\Big[\|x_{n+1}-s_n\|^2+\|s_n-s_{n-1}\|^2\Big]+\alpha^2\|s_n-s_{n-1}\|^2\nonumber\\
		&=&(1-\alpha)\|x_{n+1}-s_n\|^2-\alpha(1-\alpha)\|s_n-s_{n-1}\|^2.
	\end{eqnarray}
	Replacing $x_{n+1}$ in \eqref{17} with $s_{n+1},$ we obtain
	\begin{eqnarray}\label{18}
		\|s_{n+1}-w_n\|\geq (1-\alpha)\|s_{n+1}-s_n\|^2-\alpha(1-\alpha)\|s_n-s_{n-1}\|^2.
	\end{eqnarray}
	Putting \eqref{17} and \eqref{18} in \eqref{16} yields
	\begin{eqnarray}\label{20}
		\dfrac{1}{1-\delta}\|s_{n+1}-\bar{r}\|^2&\leq&	\dfrac{1}{1-\delta}\|w_n-\bar{r}\|^2 -\xi\Big[(1-\alpha)\|x_{n+1}-s_n\|^2-\alpha(1-\alpha)\|s_n-s_{n-1}\|^2\Big] \nonumber\\
		&&\;-\; \dfrac{\delta}{(1-\delta)^2}\Big[(1-\alpha)\|s_{n+1}-s_n\|^2-\alpha(1-\alpha)\|s_n-s_{n-1}\|^2\Big].
	\end{eqnarray}
	Moreover,
	\begin{eqnarray}\label{21}
		\|w_n-\bar{r}\|^2 &=& \|(1+\alpha)(s_n-\bar{r})-\alpha(s_{n-1}-\bar{r})\|^2\nonumber\\
		&=&(1+\alpha)\|s_n-\bar{r}\|^2-\alpha\|s_{n-1}-\bar{r}\|^2+\alpha(1+\alpha)\|s_n-s_{n-1}\|^2.
	\end{eqnarray}
	Using \eqref{21} in \eqref{20}, we have
	\begin{eqnarray*}
		\dfrac{1}{1-\delta}\|s_{n+1}-\bar{r}\|^2&\leq&\dfrac{1}{1-\delta}(1+\alpha)\|s_n-\bar{r}\|^2-\dfrac{\alpha}{1-\delta}\|s_{n-1}-\bar{r}\|^2+\dfrac{\alpha(1+\alpha)}{1-\delta}\|s_n-s_{n-1}\|^2\\
		&&\;-\;\xi(1-\alpha)\|x_{n+1}-s_n\|^2+\alpha\xi(1-\alpha)\|s_n-s_{n-1}\|^2\\
		&&\;-\; \dfrac{\delta(1-\alpha)}{(1-\delta)^2}\|s_{n+1}-s_n\|^2+ \dfrac{\alpha\delta(1-\alpha)}{(1-\delta)^2}\|s_n-s_{n-1}\|^2
	\end{eqnarray*}
	Or equivalently,
	\begin{eqnarray}\label{22}
		&&\dfrac{1}{1-\delta}\|s_{n+1}-\bar{r}\|^2-\dfrac{\alpha}{1-\delta}\|s_n-\bar{r}\|^2+\dfrac{\delta(1-\alpha)}{(1-\alpha)^2}\|s_{n+1}-s_n\|^2\nonumber\\
		&\leq&\dfrac{1}{1-\delta}\|s_n-\bar{r}\|^2-\dfrac{\alpha}{1-\delta}\|s_{n-1}-\bar{r}\|^2+\dfrac{\delta(1-\alpha)}{(1-\delta)^2}\|s_n-s_{n-1}\|^2\nonumber\\
		&&\;+\; \left[\dfrac{\alpha(1+\alpha)}{1-\delta}+\alpha\xi(1-\alpha)+\dfrac{\alpha\delta(1-\alpha)}{(1-\delta)^2}-\dfrac{\delta(1-\alpha)}{(1-\delta)^2}\right]\|s_n-s_{n-1}\|^2-\xi(1-\alpha)\|x_{n+1}-s_n\|^2.
	\end{eqnarray}
	Define $\Psi_n := \dfrac{1}{1-\delta}\|s_n-\bar{r}\|^2-\dfrac{\alpha}{1-\delta}\|s_{n-1}-\bar{r}\|^2+\dfrac{\delta(1-\alpha)}{(1-\delta)^2}\|s_n-s_{n-1}\|^2.$\\
	Then, we can write \eqref{22} as
	\begin{eqnarray}\label{23}
		\Psi_{n+1}&\leq&\Psi_n+ \left[\dfrac{\alpha(1+\alpha)}{1-\delta}+\alpha\xi(1-\alpha)+\dfrac{\alpha\delta(1-\alpha)}{(1-\delta)^2}-\dfrac{\delta(1-\alpha)}{(1-\delta)^2}\right]\|s_n-s_{n-1}\|^2\nonumber\\
		&&\;-\;\xi(1-\alpha)\|x_{n+1}-s_n\|^2.
	\end{eqnarray}
	Let us show that $\Psi_n\geq0.$ Using the Peter-Paul inequality, we have, for $\sigma>0,$ that
	\begin{eqnarray}\label{23a}
		\Psi_n &=& \dfrac{1}{1-\delta}\|s_n-\bar{r}\|^2-\dfrac{\alpha}{1-\delta}\|s_{n-1}-\bar{r}\|^2+\dfrac{\delta(1-\alpha)}{(1-\delta)^2}\|s_n-s_{n-1}\|^2\nonumber\\
		&\geq&\dfrac{1}{1-\delta}\|s_n-\bar{r}\|^2-\dfrac{\alpha}{1-\delta}\left(1+\dfrac{1}{\sigma}\right)\|s_n-\bar{r}\|^2-\dfrac{\alpha(1+\sigma)}{1-\delta}\|s_n-s_{n-1}\|^2+\dfrac{\delta(1-\alpha)}{(1-\delta)^2}\|s_n-s_{n-1}\|^2\nonumber\\
		&=&\dfrac{1}{1-\delta}\left(1-\alpha\Big(1+\dfrac{1}{\sigma}\Big)\right)\|s_n-\bar{r}\|^2+\dfrac{1}{1-\delta}\left(\dfrac{\delta(1-\alpha)}{1-\delta}-\alpha(1+\sigma)\right)\|s_n-s_{n-1}\|^2\\
		&\geq& 0, \nonumber
	\end{eqnarray}
	since $0\leq\alpha<\dfrac{\sigma}{1+\sigma}$ and $\dfrac{\alpha(1+\sigma)}{1+\alpha\sigma}<\delta.$\\
	
	\noindent We further show that $\dfrac{\alpha(1+\alpha)}{1-\delta}+\alpha\xi(1-\alpha)+\dfrac{\alpha\delta(1-\alpha)}{(1-\delta)^2}-\dfrac{\delta(1-\alpha)}{(1-\delta)^2}<0.$\\
	Clearly, if $\alpha=0$ and $0<\delta<1,$ the result follows immediately. Suppose $\alpha\neq0,$ that is, $\alpha\in(0,1)$ and $0<\delta<1,$ then, from Assumption \ref{Assumconditions}, we have that
	$$\dfrac{\alpha(1+\alpha)+2\alpha\xi+1-\sqrt{\alpha^4+2\alpha^3+3\alpha^2+4\alpha\xi+2\alpha+1}}{2\alpha\xi}<\delta.$$
	That is,
	\begin{eqnarray}\label{24}
		\left[\delta-\left(\dfrac{\alpha(1+\alpha)+2\alpha\xi+1-\sqrt{\Big(\alpha(1+\alpha)+2\alpha\xi+1\Big)^2-4\alpha\xi\Big(\alpha(1+\alpha+\xi)\Big)}}{2\alpha\xi}\right)\right]>0.
	\end{eqnarray}
	Obviously,
	$$\dfrac{\alpha(1+\alpha)+2\alpha\xi+1+\sqrt{\alpha^4+2\alpha^3+3\alpha^2+4\alpha\xi+2\alpha+1}}{2\alpha\xi}>1>\delta.$$
	So,
	\begin{eqnarray}\label{25}
		\left[\delta-\left(\dfrac{\alpha(1+\alpha)+2\alpha\xi+1+\sqrt{\Big(\alpha(1+\alpha)+2\alpha\xi+1\Big)^2-4\alpha\xi\Big(\alpha(1+\alpha+\xi)\Big)}}{2\alpha\xi}\right)\right]<0.
	\end{eqnarray}
	When we multiply \eqref{24} by \eqref{25}, and group the terms together, we have
	\begin{eqnarray*}
		\alpha\xi\delta^2-\delta\Big(\alpha(1+\alpha)+2\alpha\xi+1\Big)+\alpha(1+\alpha+\xi)<0.
	\end{eqnarray*}
	Or equivalently
	\begin{eqnarray*}
		\alpha(1+\alpha)-\alpha\delta(1+\alpha)+\alpha\xi(1-2\delta+\delta^2)+\alpha\delta-\delta+\alpha\delta<0.
	\end{eqnarray*}
	This implies that 
	\begin{eqnarray*}
		(1-\delta)\Big(\alpha(1+\alpha)\Big)+\alpha\xi(1-\delta)^2+\alpha\delta-\delta(1-\alpha)<0.
	\end{eqnarray*}
	This further implies that
	\begin{eqnarray*}
		\dfrac{\alpha(1+\alpha)}{1-\delta}+\alpha\xi+\dfrac{\alpha\delta}{(1-\delta)^2}-\dfrac{\delta(1-\alpha)}{(1-\delta)^2}<0,
	\end{eqnarray*}
	and since $\alpha\in(0,1),$ the claim follows immediately. Consequently, it follows from \eqref{23} that $\Psi_n$ is a monotone non-increasing sequence, hence, $\underset{n\to\infty}{\lim}\Psi_n$ exists. Therefore, $\Psi_n$ is bounded. Accordingly, we have from \eqref{23} that
	\begin{eqnarray}\label{26}
		\underset{n\to\infty}{\lim}\|x_{n+1}-s_n\| = 0,
	\end{eqnarray}
	and
	\begin{eqnarray}\label{27}
		\underset{n\to\infty}{\lim}\|s_n-s_{n-1}\| = 0.
	\end{eqnarray}
	Observe from \eqref{23a} that $\dfrac{1}{1-\delta}\left(1-\alpha\Big(1+\dfrac{1}{\sigma}\Big)\right)\|s_n-\bar{r}\|^2\leq\Psi_n.$ But $\{\Psi_n\}$ is a bounded sequence, thus, $\{\|s_n-\bar{r}\|^2\}$ is bounded. Hence, the sequence $\{s_n\}$ is bounded. Consequently, $\{w_n\}$ is also bounded, since $w_n = s_n+\alpha(s_n-s_{n-1}).$ Again, since $\{s_n\}$ and $\{w_n\}$ are bounded, it follows that $\{x_n\}$ is also bounded.
\end{proof}
\begin{theorem}\label{Theo pox contr}
	Let Assumptions \ref{ASSUMproxcon} and \ref{Assumconditions} be fulfilled. Then the iterative sequence $\{x_n\}$ generated by Algorithm \ref{ALG proxcon} converges weakly to a point in $\Delta(\mathcal{T}; g).$
\end{theorem}
\begin{proof}\noindent\\
	Recall that $w_n = s_n+\alpha(s_n-s_{n-1}),$  therefore, using \eqref{27}, we get
	$$\underset{n\to\infty}{\lim}\|w_n-s_n\| = \alpha\underset{n\to\infty}{\lim}\|s_n-s_{n-1}\| = 0.$$
	Again, using \eqref{26} and the last equation, we obtain
	\begin{eqnarray}\label{28}
		\|x_{n+1}-w_n\|\leq\|x_{n+1}-s_n\|+\|s_n-w_n\|\longrightarrow 0, \ \ n\to\infty.
	\end{eqnarray}
	Given \eqref{LMPB} and \eqref{28}, and that $\underset{n\to\infty}{\lim}\lambda_n$ exists, we can deduce that
	$$\|w_n-y_n\|\longrightarrow0, \ \ n\to\infty.$$
	
	\noindent Let $r_\omega\in\mathcal{H}$ be a weak sequential cluster point of $\{x_n\}.$ We show that $r_\omega\in\Delta(\mathcal{T}; g).$ Due to the boundedness of $\{x_n\},$ there is a subsequence $\{x_{n_k}\}$ of $\{x_n\}$ such that $x_{n_k}\rightharpoonup r_\omega$ as $k\to\infty.$ Also, it follows that $\{w_{n_k}\}, \ \{y_{n_k}\}$ and $\{s_{n_k}\}$ converge weakly to $r_\omega.$ But $y_{n_k}\in\mathcal{C}$ and $\mathcal{C}$ is weakly closed, therefore, $r_\omega\in\mathcal{C}.$\\
	
	\noindent Let $z\in\mathcal{C}$ be an arbitrarily fixed point. Then we obtain from \eqref{5prox} that
	\begin{eqnarray*}
		\lambda_{n_k}\Big(g(y_{n_k})-g(z)\Big)\leq\langle y_{n_k}-z, w_{n_k}-y_{n_k}-\lambda_{n_k}\mathcal{T}w_{n_k}\rangle.
	\end{eqnarray*}
	Or equivalently
	\begin{eqnarray}\label{29}
		\lambda_{n_k}\Big(g(y_{n_k})-g(z)\Big)&\leq&\langle y_{n_k}-w_{n_k}+\lambda_{n_k}\mathcal{T}w_{n_k} , z-y_{n_k}\rangle\nonumber\\
		&=& \langle y_{n_k}-w_{n_k}, z-y_{n_k}\rangle+\lambda_{n_k}\langle\mathcal{T}w_{n_k}, z-w_{n_k}\rangle\nonumber\\
		&&\; +\; \lambda_{n_k}\langle \mathcal{T}w_{n_k}, w_{n_k}-y_{n_k}\rangle\nonumber\\
		&\leq&  \langle y_{n_k}-w_{n_k}, z-y_{n_k}\rangle + \lambda_{n_k}\langle\mathcal{T}z, z-w_{n_k}\rangle +  \lambda_{n_k}\langle \mathcal{T}w_{n_k}, w_{n_k}-y_{n_k}\rangle,
	\end{eqnarray}
	where we obtained the last inequality using the fact that $\mathcal{T}$ is monotone ($\mathcal{T}$ satisfies the generalized monotonicity assumption). Letting $k\to+\infty$ in \eqref{29}, observing that $\|y_{n_k}-w_{n_k}\|\longrightarrow0,$ items (B$_2$) and (B$_3$) of Assumption \ref{ASSUMproxcon} hold, and $\underset{n\to\infty}{\lim}\lambda_n = \lambda^*>0,$ we obtain
	\begin{eqnarray}\label{30}
		\langle \mathcal{T}z, z-r_\omega\rangle + g(z)-g(r_\omega)\geq0.
	\end{eqnarray}
	Let $u\in\mathcal{C}$ be arbitrarily chosen. Then, for every $t\in(0,1), \ r_\omega\in\mathcal{C},$ define $z_t := tu+(1-t)r_\omega.$ Thus, $z_t\in\mathcal{C}$ since $\mathcal{C}$ is convex. Replacing $z$ in \eqref{30} by $z_t,$ we obtain 
		\begin{eqnarray*}
		\langle \mathcal{T}z_t, z_t-r_\omega\rangle + g(z_t)-g(r_\omega)\geq0.
	\end{eqnarray*}
	Hence, by the convexity of $g,$ we get
	\begin{eqnarray}\label{31}
		\langle \mathcal{T}z_t, u-r_\omega\rangle + g(u)-g(r_\omega)\geq0.
	\end{eqnarray}
	Taking limit as $t\longrightarrow 0$ in \eqref{31} and using the fact that $\mathcal{T}$ is sequentially weakly continuous, we get
	\begin{eqnarray*}
		\langle \mathcal{T}r_\omega, u-r_\omega\rangle + g(u)-g(r_\omega)\geq0.
	\end{eqnarray*}
	And because $u\in\mathcal{C}$ is arbitrarily chosen, we have that $r_\omega\in\Delta(\mathcal{T}; g).$\\
	
	\noindent Next, we show that $\underset{n\to\infty}{\lim}\|x_n-r_\omega\|$ exists for any $r_\omega\in\Delta(\mathcal{T}; g).$\\
	Let us define 
	\begin{eqnarray*}
		\ell_n &:=& 2\langle s_{n-1}-s_n, s_n-r_\omega\rangle+ \|s_{n-1}-s_n\|^2,\\
	m_n &:=& -\dfrac{\delta(1-\alpha)}{(1-\delta)^2}\|s_n-s_{n-1}\|^2.
	\end{eqnarray*}
	Observe from Lemma \ref{LM con1} that
	\begin{eqnarray*}
		\|s_{n-1}-r_\omega\|^2 = \|s_{n-1}-s_n\|^2+2\langle s_{n-1}-s_n, s_n-r_\omega\rangle+\|s_n-r_\omega\|^2.
	\end{eqnarray*}
	Thus, it follows that
	\begin{eqnarray}\label{32}
		\dfrac{1-\alpha}{1-\delta}\|s_n-r_\omega\|^2 = \Psi_n+ \ell_n+\dfrac{\alpha}{1-\delta}m_n.
	\end{eqnarray}
	But $\underset{n\to\infty}{\lim}\|s_{n-1}-s_n\| = 0,$ and $\{s_n\}$ is bounded, hence,  $\underset{n\to\infty}{\lim}\ell_n =  \underset{n\to\infty}{\lim} m_n =0.$\\
	Recall that  $\underset{n\to\infty}{\lim}\Psi_n$ exists, it follows from \eqref{32} that  $$\underset{n\to\infty}{\lim}\|s_n-r_\omega\| \  \mbox{exists for any} \ r_\omega\in\Delta(\mathcal{T}; g).$$
	So, we obtain from 
	\begin{eqnarray*}
		\|x_n-r_\omega\|^2 = \|x_n-s_n\|^2 + 2\langle x_n-s_n, s_n-r_\omega\rangle + \|s_n -r_\omega\|^2
	\end{eqnarray*}
	that  $\underset{n\to\infty}{\lim}\|x_n-r_\omega\|$ exists for each $r_\omega\in\Delta(\mathcal{T}; g).$ Hence, we conclude from Lemma \ref{Opial Lemma} that $\{x_n\}$ converges weakly to an element in $\Delta(\mathcal{T}; g).$
\end{proof}

\section{Numerical Examples}\label{numerics}
\noindent In this section, we show the numerical results of our Algorithm \ref{ALG proxcon}. In particular, we use the following examples to compare the performance of our Algorithm \ref{ALG proxcon} (denoted here as Alg \ref{ALG proxcon}) with the methods studied by Kim \cite{Kim} (denoted here as Alg Kim), Maingé \cite{Maingé} (denoted here as Alg Main), Dong \textit{et al.} \cite{Dong Cho} (denoted here as Alg Don) and Jolaoso \textit{et al. }\cite{JOLAOSO} (denoted here as Alg Jol). 

\begin{example} \label{EX1} \cite[Example 5.2]{JOLAOSO}
	\noindent Let $A = [0,1]\times[0,1]\subset\mathbb{R}^2$ and $\mathcal{C} = \{x\in\mathbb{R}^3 ~ \vert~ Mx+d\in A\},$ where
	$$M = \begin{bmatrix}
		1 & 2 & 1\\
		1 & 1 & 1
	\end{bmatrix} \ \ \mbox{and}\ d = \begin{bmatrix}
	1/2\\
	1/2
	\end{bmatrix}.$$ 
	Define $\mathcal{T}(x):= M^TG(Mx+d),$ with
	\begin{eqnarray*}
		G(x)= \begin{cases}
			\left(\dfrac{-t}{1+t}, \dfrac{-1}{1+t}\right), \ \ \ (x_1, x_2)\neq 0,\\ \\
			(0, -1), \ \ \ \ \ \ \ \ \ \ \ \ \ \ \ (x_1, x_2) = 0,
		\end{cases}
	\end{eqnarray*}
	and $$t = \dfrac{x_1+\sqrt{x_1^2+4x_2}}{2}, \ \ x_1, x_2\in\mathbb{R}.$$
	Take $g = \delta_\mathcal{C}.$
\end{example}

\hfill

\begin{example} \cite[Example 5.3]{JOLAOSO}\label{Ex2}
	\noindent Let $\mathcal{C} = [3,5]\times[3,5],$
	\begin{eqnarray*}
		g((x_1, x_2)) = \begin{cases}
			(x_1^2, x_2^2), \ \ \ \ (x_1, x_2)\in\mathcal{C},\\ \\
			+\infty, \ \ \ \ \ \ \ \  (x_1, x_2) \notin \mathcal{C},
		\end{cases}
	\end{eqnarray*}
	and $\mathcal{T}(x) := (4-x_1, 4-x_2),$ for $(x_1, x_2)\in\mathbb{R}^2.$ Then $\mathcal{T}$ is $g-$pseudomonotone and $\Delta(\mathcal{T}; g) = \{(3,3)\}.$
\end{example}

\begin{example} \cite[Example 5.1]{Izuchukwu}\label{Ex3}
	Let $B, D\in\mathbb{R}^{n\times n}$ be symmetric and positive definite matrices. For MVI\eqref{MVI non}, let $g(x) = x^TBx$ and $\langle \mathcal{T}\bar{x}, x-\bar{x}\rangle = x^TD(x-\bar{x}).$ Let the largest eigenvalue of $D$ and the smallest eigenvalue of $B$ be denoted respectively by $\rho(D)$ and $\eta(B).$ Then both $\rho(D)$ and $\eta(B)$ are positive since $D$ and $B$ are positive definite.
	
\end{example}
\noindent  During the computations, we randomly generate the starting points $w_{-1}, w_{-2}, x_0, x_{-1}$ for $n=20,50,100.$ We choose our control parameters as follows:
\begin{itemize}
	\item Alg \ref{ALG proxcon}: $\alpha = 0.5,$ $\delta = 0.9,$ $\lambda = \dfrac{0.99}{2\rho(B)},$ $ \gamma = 1.5,$ $\theta = 0.4,$ $\xi = \dfrac{7}{3},$ and $\sigma = 1.5.$
	\item  Alg Main: $a = 1,$ $c= 2,$ $b= 0.5,$ $a_1 = 0.5,$ $a_2 = 0.9,$ $\bar{c} = 1,$ and $ \lambda = \dfrac{0.99}{2\rho(B)}.$
	\item  Alg Kim: $ \lambda = \dfrac{0.99}{2\rho(B)}.$
	\item  Alg Don: $\alpha_n = 0.3-\dfrac{1}{5(n+1)^2},$ $\tau = \dfrac{0.99}{2\rho(B)},$ $\gamma = \dfrac{58}{477},$ $\delta = 0.9,$ $\alpha = 0.4,$ and $\sigma = 0.2.$ 
	\item  Alg Jol: $\theta = 0.3,$ $\gamma= 0.5,$ and $\sigma = 0.5.$
\end{itemize}
All computations are performed using MATLAB R2023b, which runs on a personal computer featuring an Intel(R) Core(TM) i5-10210U CPU at 2.11 GHz and 8.00 GB of RAM. Throughout the experiments, we define $TOL_n=\|x_{n+1}-x_n\|$. Then, using the stopping criterion $TOL_n<\epsilon$ for the iteration methods, where $\epsilon$ is the predetermined error. It is worthy to note that $TOL_n=0$ meaning that $x_n$ is a solution of MVI \eqref{MVI non}.\\

\hfill

\noindent By using Example \ref{EX1} and the control parameters earlier provided, we obtain numerical simulations across various dimensions. The comprehensive results from these simulations are systematically compiled and presented in Table \ref{tb1}, and the corresponding visual representation of these findings is displayed in Figure \ref{fig1}, as shown below.
	\begin{table}[H]
	\centering
	\caption{Results of the Numerical Simulations for Different Dimensions}\label{tb1}
	\resizebox{\columnwidth}{!}{\begin{tabular}{|c |cc| cc| cc | cc |cc|}\hline
			\multicolumn{11}{|c|}{ Numerical Results for $n=20, 50$ and $100$ in Example \ref{EX1}}\\ \hline
			& \multicolumn{2}{|c|}{Alg 3.3 }  & \multicolumn{2}{|c|} { Alg Kim }  & \multicolumn{2}{|c} { Alg Jol}   & \multicolumn{2}{|c|} { Alg Main}&\multicolumn{2}{|c|} { Alg Don}    \\\hline
			n & Iter  &  CPU time (sec.) & Iter  &  CPU time (sec.) & Iter  &  CPU time (sec.) & Iter  &  CPU time (sec.) & Iter  &  CPU time (sec.) \\
			20 & 90  & 0.0525 & 120 & 0.0762 & 155& 0.0941 & 270& 0.3418 & 290& 0.6418 \\
			50 & 78 & 0.0428 & 85 & 0.0687 & 139 & 0.0898  & 150 & 0.0759& 200& 0.0818 \\
			100 & 89 & 0.0555 & 101 & 0.0691 & 145 & 0.0871  & 170 & 0.0932& 250& 0.1418\\ 
			\hline
	\end{tabular}}
\end{table}

\hfill

\begin{figure}[H] 
	\includegraphics[{width=7cm, height=7.5cm}]{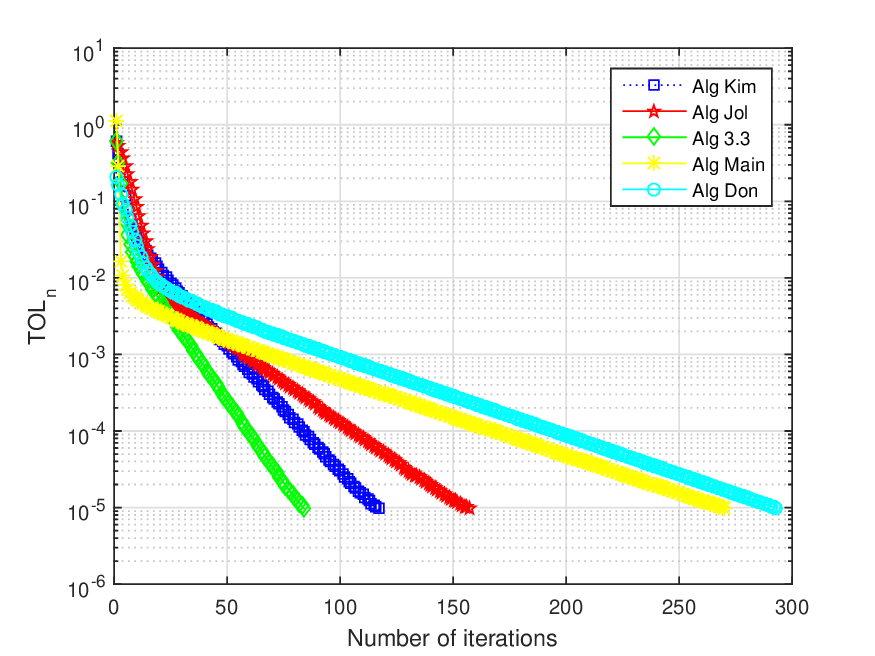}  
	\includegraphics[{width=7cm, height=7.5cm}]{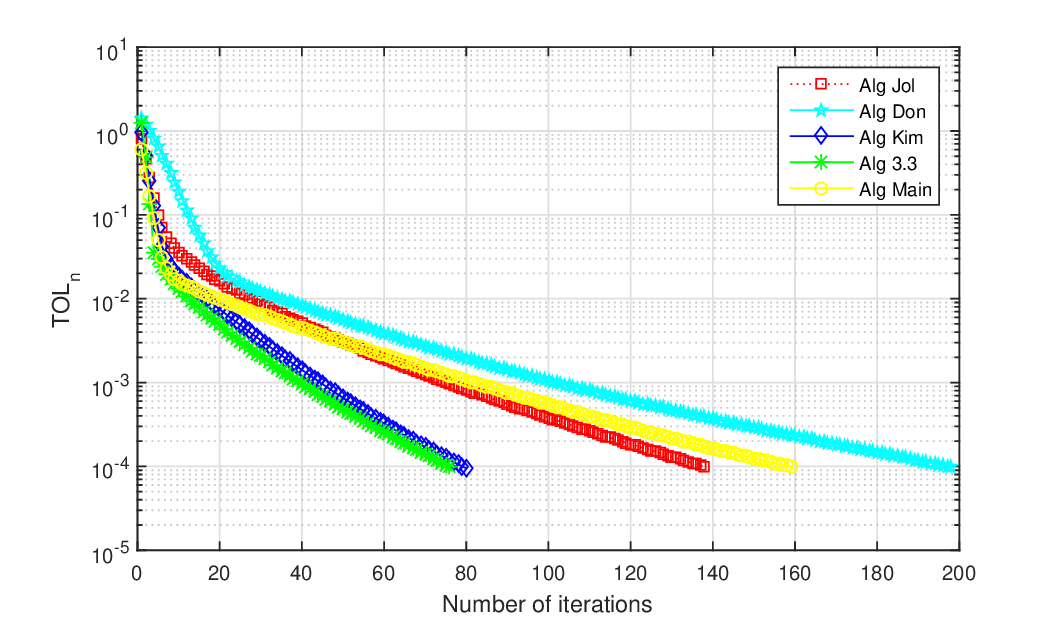} 
	\includegraphics[{width=7cm, height=7.5cm}]{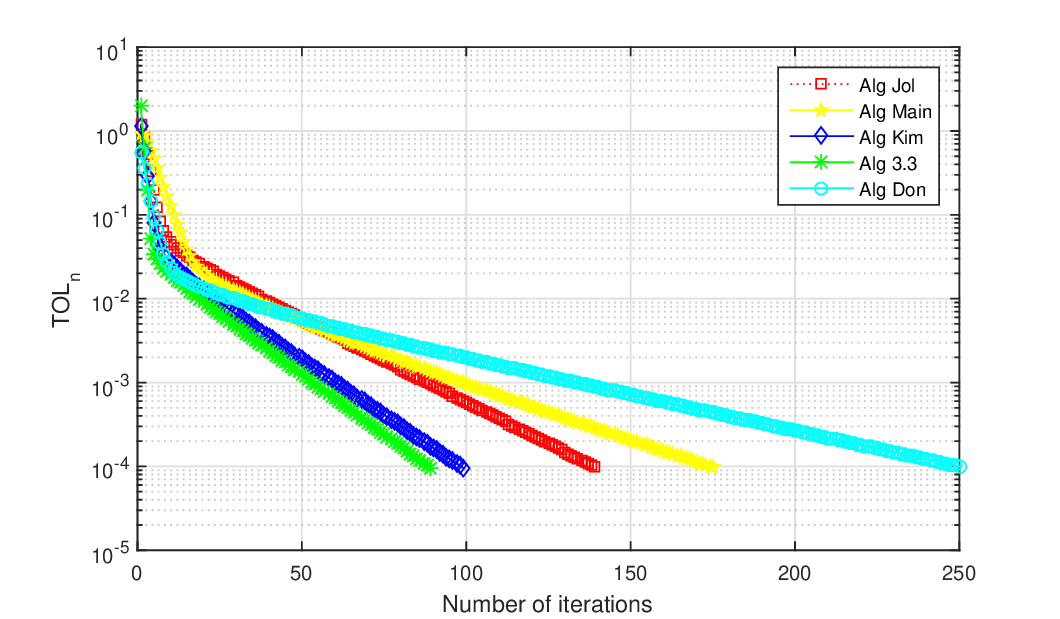} 
	\caption{Graph of the  Iterates for Example \ref{EX1} when the Dimensions $n=20 $, $n=50$ and $n=100$ }\label{fig1}
\end{figure}

\noindent Based on Example \ref{Ex2} and the specified control parameters, numerical simulations were carried out for different dimensions. The resulting comparative data is compiled in Table \ref{tb2}, and the visual findings are simultaneously presented in Figure \ref{fig2} below.

	\begin{table}[H]
	\centering
	\caption{Results of the Numerical Simulations for Different Dimensions}\label{tb2}
	\resizebox{\columnwidth}{!}{\begin{tabular}{|c |cc| cc| cc | cc |cc|}\hline
			\multicolumn{11}{|c|}{ Numerical Results for $n=20, 50$ and $100$ in Example \ref{Ex2}}\\ \hline
			& \multicolumn{2}{|c|}{Alg 3.3 }  & \multicolumn{2}{|c|} { Alg Kim }  & \multicolumn{2}{|c} { Alg Jol}   & \multicolumn{2}{|c|} { Alg Main}&\multicolumn{2}{|c|} { Alg Don}    \\\hline
			n & Iter  &  CPU time (sec.) & Iter  &  CPU time (sec.) & Iter  &  CPU time (sec.) & Iter  &  CPU time (sec.) & Iter  &  CPU time (sec.) \\
			20 & 160  & 0.1623 & 190 & 0.2468 & 255& 0.4944 & 400& 0.5419 & 1030& 1.4418 \\
			50 & 75 & 0.0627 & 100 & 0.0887 & 220 & 0.1898  & 400 & 0.2759& 1400& 1.5818 \\
			100 & 140 & 0.0855 & 220 & 0.1691 & 300 & 0.2873  & 353 & 0.3972& 450& 0.6418\\ 
			\hline
	\end{tabular}}
\end{table}
\begin{figure}[H] 
	\includegraphics[{width=7cm, height=7.5cm}]{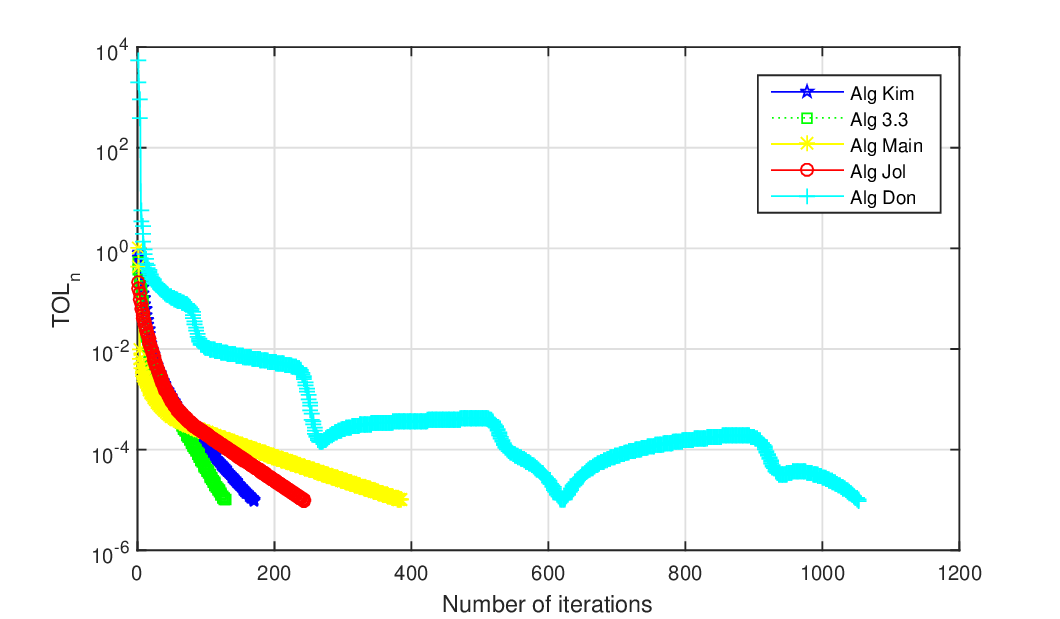}  
	\includegraphics[{width=7cm, height=7.5cm}]{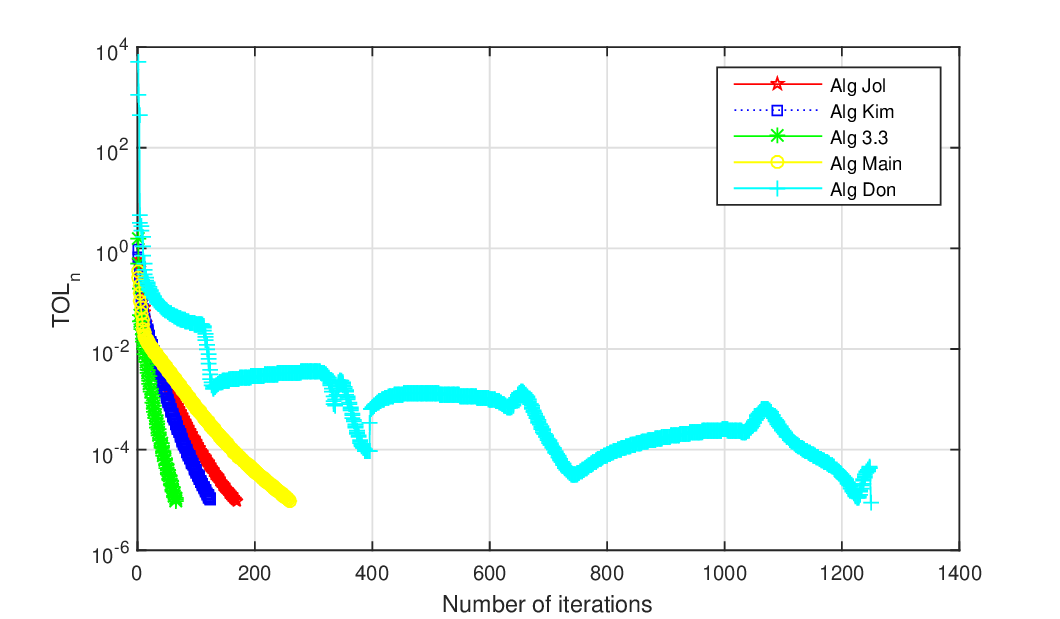} 
	\includegraphics[{width=7cm, height=7.5cm}]{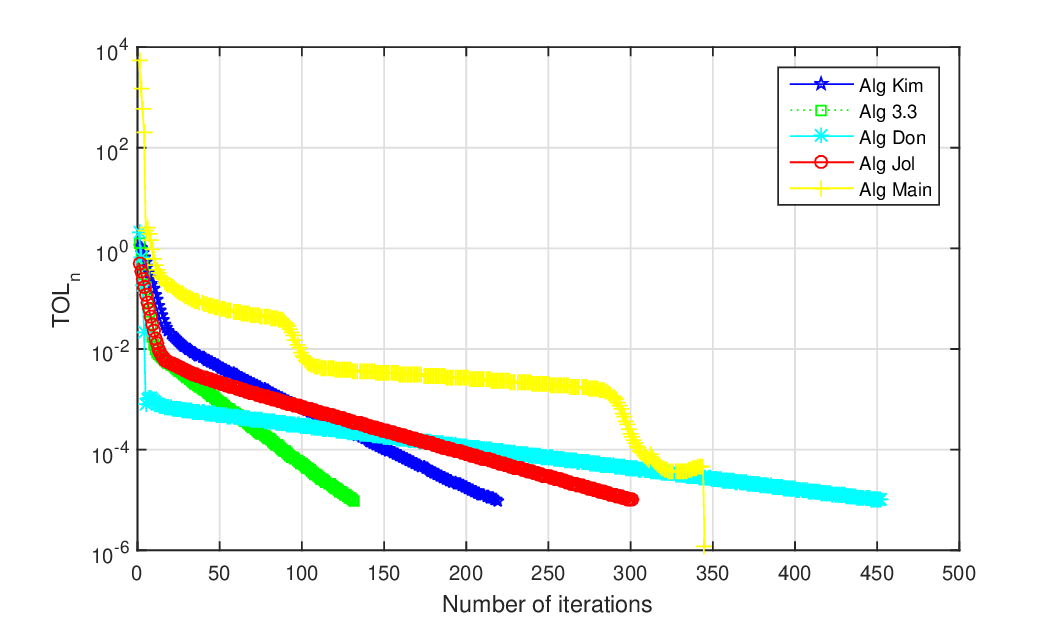} 
	\caption{Graph of the  Iterates for Example \ref{Ex2} when the Dimensions $n=20 $, $n=50$ and $n=100$}\label{fig2}
\end{figure}

\noindent In accordance with Example \ref{Ex3} and the previously defined control parameters, the numerical simulations were performed for various dimensions. The derived comparative results are compiled within Table \ref{tb3}, and the corresponding visual data is concurrently presented in Figure \ref{fig3} as follows:

		\begin{table}[H]
		\centering
		\caption{Results of the Numerical Simulations for Different Dimensions}\label{tb3}
		\resizebox{\columnwidth}{!}{\begin{tabular}{|c |cc| cc| cc | cc |cc|}\hline
				\multicolumn{11}{|c|}{ Numerical Results for $n=20, 50$ and $100$ in Example \ref{Ex3}}\\ \hline
				& \multicolumn{2}{|c|}{Alg 3.3 }  & \multicolumn{2}{|c|} { Alg Kim }  & \multicolumn{2}{|c} { Alg Jol}   & \multicolumn{2}{|c|} { Alg Main}&\multicolumn{2}{|c|} { Alg Don}    \\\hline
				n & Iter  &  CPU time (sec.) & Iter  &  CPU time (sec.) & Iter  &  CPU time (sec.) & Iter  &  CPU time (sec.) & Iter  &  CPU time (sec.) \\
				20 & 35  & 0.0125 & 46 & 0.0362 & 62& 0.0541 & 100& 0.0718 & 110& 0.7918 \\
				50 & 50 & 0.0324 & 55 & 0.0487 & 110 & 0.0596  & 115 & 0.0589& 210& 0.0818 \\
				100 & 47 & 0.0256 & 65 & 0.0399 & 105 & 0.0578  & 115 & 0.0632& 124& 0.0814\\ 
				\hline
		\end{tabular}}
	\end{table}
\begin{figure}[H] 
	\includegraphics[{width=7cm, height=7.5cm}]{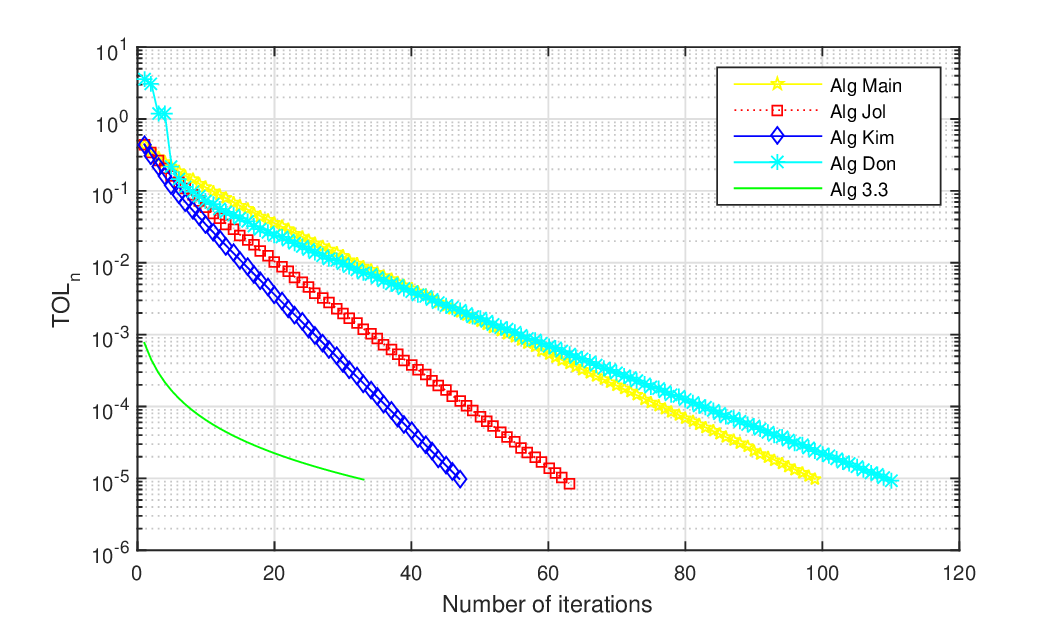}  
	\includegraphics[{width=7cm, height=7.5cm}]{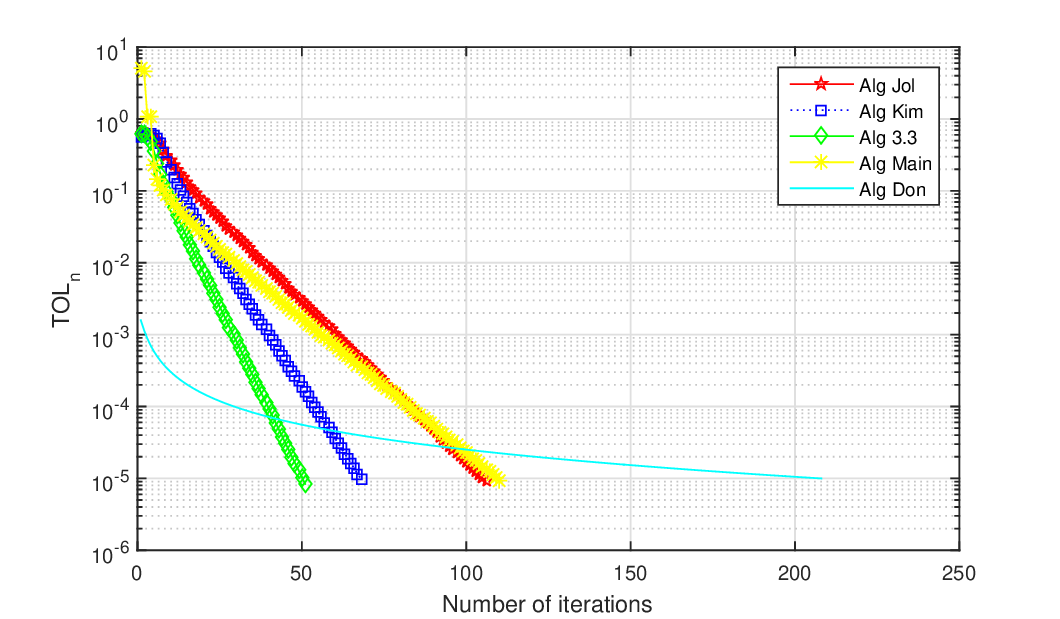} 
	\includegraphics[{width=7cm, height=7.5cm}]{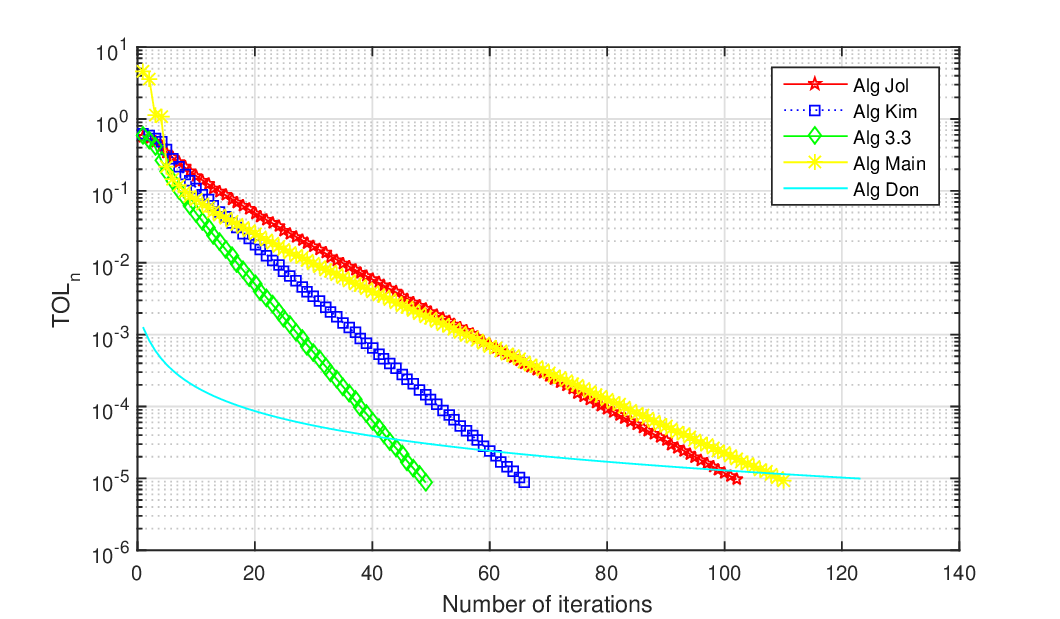} 
	\caption{Graph of the  Iterates for Example \ref{Ex3} when the Dimensions $n=20 $, $n=50$ and $n=100$}\label{fig3}
\end{figure}

	\hfill
	\begin{remark}
		It can be noticed from the numerical results presented above that the incorporation of an inertial term and two correction terms indeed accelerates the proximal and contraction method, thereby outperforming other methods in the literature without such effects. In terms of the number of iterations and CPU time, we can also see that our proposed Algorithm \ref{ALG proxcon} is faster than other methods in \cite{Dong Cho,JOLAOSO,Kim,Maingé}.
	\end{remark}
	
	\hfill
	
	\section{Conclusion}
	\noindent In this paper, we introduced a proximal and contraction method for solving mixed variational inequality problem in a real Hilbert space. Our method combines two correction terms with an inertial term. This is a novel speeding technique for solving MVIPs. Under standard conditions, we obtain a weak convergence result of our proposed method. Different numerical examples presented show that our method has a competitive advantage over other related methods. Thanks to the relaxed inertial and correction terms. Our further research plans to investigate the strong convergence of a three-operator monotone inclusion problem by incorporating relaxed inertial and correction terms.
	
	\section{Declaration}
	\subsection{Acknowledgement:} The authors appreciate the support provided by their institutions.
	
	\subsection{Funding:}  No funding received.
	
	\subsection{Use of AI.:}
	The authors declare that they did not use AI to generate any part of the paper.
	\subsection{Availability of data and material:}  Not applicable.
	\subsection{Competing interests:}
	The authors declare that they have no competing interests.
	\subsection{Authors' contributions:}
	All authors worked equally on the results and approved the final manuscript.
	
	\hfill

\end{document}